\documentclass[oneside,11pt]{amsart}
\usepackage[pdftex]{color}
\usepackage{array}
\usepackage{mathtools}
\usepackage{amstext}
\usepackage{amsthm}
\usepackage{amssymb}
\usepackage[pdftex]{graphicx}
\usepackage[unicode=true,pdfusetitle,
 bookmarks=true,bookmarksnumbered=false,bookmarksopen=false,
 breaklinks=false,pdfborder={0 0 0},pdfborderstyle={},backref=false,colorlinks=true]
 {hyperref}

\makeatletter

\providecommand{\tabularnewline}{\\}

\theoremstyle{plain}
\newtheorem{thm}{\protect\theoremname}
\theoremstyle{definition}
\newtheorem{example}[thm]{\protect\examplename}
\theoremstyle{plain}
\newtheorem{cor}[thm]{\protect\corollaryname}
\theoremstyle{plain}
\newtheorem{prop}[thm]{\protect\propositionname}
\theoremstyle{plain}
\newtheorem{lem}[thm]{\protect\lemmaname}

\makeatother

\providecommand{\corollaryname}{Corollary}
\providecommand{\examplename}{Example}
\providecommand{\lemmaname}{Lemma}
\providecommand{\propositionname}{Proposition}
\providecommand{\theoremname}{Theorem}

\begin{document}
\global\long\def\mathistOperatorAd{\operatorname{Ad}}%

\global\long\def\mathistOperatorAD{\operatorname{\boldsymbol{\mathrm{Ad}}}}%

\global\long\def\mathistOperatorarccosh{\operatorname{arccosh}}%

\global\long\def\mathistOperatorarccot{\operatorname{arccot}}%

\global\long\def\mathistOperatorarcsinh{\operatorname{arcsinh}}%

\global\long\def\mathistOperatorarctanh{\operatorname{arctanh}}%

\global\long\def\mathistOperatorArg{\operatorname{Arg}}%

\global\long\def\mathistOperatord{\mathop{}\!\mathrm{d}}%

\global\long\def\mathistOperatordiam{\operatorname{diam}}%

\global\long\def\mathistOperatoressinf{\operatorname{ess.inf}}%

\global\long\def\mathistOperatoresssup{\operatorname{ess.sup}}%

\global\long\def\mathistOperatorExt{\operatorname{Ext}}%

\global\long\def\mathistOperatorgrad{\operatorname{grad}}%

\global\long\def\mathistOperatorhol{\operatorname{hol}}%

\global\long\def\mathistOperatorHyp{\operatorname{Hyp}}%

\global\long\def\mathistOperatorIm{\operatorname{Im}}%

\global\long\def\mathistOperatorLength{\operatorname{Length}}%

\global\long\def\mathistOperatorLn{\operatorname{Ln}}%

\global\long\def\mathistOperatorRe{\operatorname{Re}}%

\global\long\def\mathistOperatorRes{\operatorname{Res}}%

\global\long\def\mathistOperatorsupp{\operatorname{supp}}%

\global\long\def\mathistOperatortr{\operatorname{tr}}%

\global\long\def\mathistOperatortw{\operatorname{tw}}%

\global\long\def\mathistOperatorVol{\operatorname{Vol}}%

\title{Bi-geodesic mappings between hyperbolic surfaces with boundary}
\author{Wen Yang}
\address{Hunan University, School of Mathematics, Lushan Road (S), Yuelu District,
Changsha, Zip Code: 410082, P.R. China and Sun Yat-sen University,
School of Mathematics, Xingang Xi Road, Guangzhou, Zip Code: 510275,
P. R. China}
\email{yang-wen@139.com}
\begin{abstract}
It is proved that a bijection between two compact hyperbolic surfaces
with boundary is an isometry if it and its inverse map each geodesic
onto some geodesic.
\end{abstract}

\maketitle

\section{Introduction}

As is well known, if $f:M\rightarrow N$ is an isometry, then $f$
and $f^{-1}$ map each geodesic onto some geodesic. Conversely, suppose
that $f:M\rightarrow N$ is a \emph{bi-geodesic mapping}, i.e., $f$
is a bijection such that $f$ and $f^{-1}$ map each geodesic onto
some geodesic. Is $f$ always an isometry? The answer is no. Many
linear transformations of Euclidean spaces are counterexamples. Here
is a non-trivial counterexample:
\begin{example}
Let $\Delta=\{z\in\mathbb{C}:|z|<1\}$ be the unit disk equipped with
the Hilbert metric. As illustrated in Figure \ref{fig:counterexample},
we connect $0$ to $e^{\pi i/3}$ and to $1$ by two copies of the
minor arc that connects $1$ to $e^{\pi i/3}$. These arcs bound a
domain $D$. Consider the Euclidean rotation
\[
f:D\rightarrow D,\quad z\mapsto e^{2\pi i/3}\left(z-\frac{\sqrt{3}}{3}e^{\pi i/6}\right)+\frac{\sqrt{3}}{3}e^{\pi i/6},
\]
which rotates $D$ by an angle of $2\pi/3$.

\begin{figure}[h]
\noindent \begin{centering}
\includegraphics{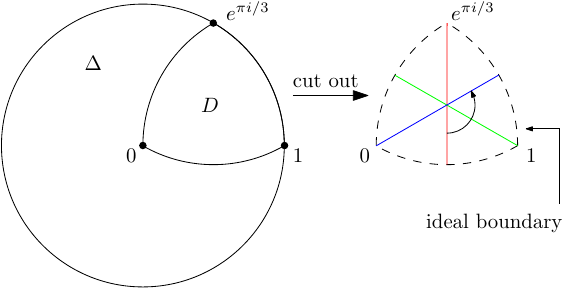}
\par\end{centering}
\caption{\label{fig:counterexample}A non-trivial bi-geodesic mapping.}
\end{figure}

The mapping $f:D\rightarrow D$ has the following properties:

(1) $f$ and $f^{-1}$ are diffeomorphisms that map each geodesic
onto some geodesic;

(2) We can naturally extend $f$ to a self-homeomorphism of $\overline{D}$;

(3) The domain $D$ is equipped with a hyperbolic metric and is neither
compact nor complete;

(4) $f$ maps many geodesics of infinite length to geodesics of finite
length;

(5) $f$ is not an isometry, since it is not even a dilation.
\end{example}

The above example shows that a general bi-geodesic mapping can be
very complicated. Nevertheless, the bi-geodesic mappings between some
spaces are trivial. In 1847 Von Staudt \cite{staudt1847geometrie}
proved the Fundamental Theorem of Affine Geometry: A bijective self-mapping
of a Euclidean space is affine if it maps lines to lines. Since then,
Von Staudt's theorem has been generalized to various settings by Beltrami,
Levi-Civita, Weyl, Sinyukov, Mike\v{s}, and others (see \cite{beltrami1868,weyl1921,lenz1958einige,chubarev1999fundamental,jeffers2000lost,li2005transformations,fu2009isometries,yao2011fundamental,mikes2015differential,li2016fundamental,artstein2017fundamental,shulkin2017fundamental,Lo2024rigidity}).

In a recent joint work \cite{liu2019geodesic} with L. Liu, we have
studied the bi-geodesic mappings between pairs of pants. The aim of
this paper is to continue the work of \cite{liu2019geodesic} and
establish:
\begin{thm}
\label{Theorem:main}Let $R_{i}$ be a compact hyperbolic surface
of genus $g_{i}$ with $n_{i}\geqslant1$ geodesic boundary components,
and $S_{i}$ the interior of $R_{i}$, $i=1,2$. Suppose that $f:S_{1}\rightarrow S_{2}$
is a bijection such that $f$ and $f^{-1}$ map each geodesic onto
some geodesic. Then $f$ is an isometry.
\end{thm}

Since a geodesic $\tau$ in a hyperbolic surface $R$ with boundary
is a boundary component if and only if every geodesic $\rho\neq\tau$
in $R$ intersects $\tau$ at most twice, by Theorem \ref{Theorem:main}
we obtain:
\begin{cor}
Let $R_{i}$ be a compact hyperbolic surface of genus $g_{i}$ with
$n_{i}\geqslant1$ geodesic boundary components, $i=1,2$. Suppose
that $f:R_{1}\rightarrow R_{2}$ is a bijection such that $f$ and
$f^{-1}$ map each geodesic onto some geodesic. Then $f$ is an isometry.
\end{cor}

Note that $f$ is assumed to be a bijection but not a homeomorphism,
and all the geodesics are regarded as point sets but not mappings.
We cannot lift $f$ to the universal covering space at once, hence
the classical proof of the Fundamental Theorem of Affine Geometry
does not work here. In addition, the geodesics here are more complex
than that in a pair of pants or a torus. In order to overcome these
difficulties, we develop the idea in \cite{liu2019geodesic} and introduce
new properties preserved by $f$ to carefully classify the geodesics
in terms of intersections with geodesics.

\section{Preliminaries}

\subsection{Hyperbolic surfaces and geodesics}

Here we recall some definitions and results from the theory of hyperbolic
surfaces, see \cite{casson1988automorphisms}, \cite{farb2011primer}
or \cite{imayoshi1992introduction} for more details.

It is well known that a smooth curve $\gamma$ is called a geodesic
if $\ddot{\gamma}=0$. A point $p$ is called a \emph{self-intersection}
of $\gamma$ if $\gamma$ transverses to itself at $p$.

However, throughout this paper, the word ``geodesic'' always means
the image of a complete geodesic. In other words, we say a subset
of a hyperbolic surface is a \emph{geodesic} if it is the image of
a locally shortest curve which cannot be extended to any other locally
shortest curve.

We call a surface, obtained from a closed surface of genus $g$ by
deleting mutually disjoint $n$ disks and $m$ points, a \emph{surface
of finite type}.
\begin{prop}[{\cite[Proposition 16]{liu2019geodesic}}]
\label{prop:isometry}A homeomorphism between the interiors of two
hyperbolic surfaces of finite type is an isometry if it maps each
geodesic onto some geodesic.
\end{prop}

\begin{lem}[Collar lemma, see Lemma 13.6 in \cite{farb2011primer}]
Let $\gamma$ be a simple closed geodesic on a hyperbolic surface
$X$. Then $C_{\gamma}=\{x\in X:d(x,\gamma)\leqslant w\}$ is an embedded
annulus, where $w$ is given by
\[
w=\sinh^{-1}\left(\frac{1}{\sinh(l(\gamma)/2)}\right).
\]
We call $C_{\gamma}$ the \emph{collar} around $\gamma$.
\end{lem}

Let $X$ be a hyperbolic surface of finite type and $\tilde{X}$ the
universal covering space of $X$.

Let $W$ be a manifold. Suppose that $F:W\rightarrow X$ is a continuous
map and $\tilde{F}:W\rightarrow\tilde{X}$ is a lift of $F$. We call
$\tilde{\tau}=\tilde{F}(W)$ a \emph{lift} of $\tau=F(W)$ to $\tilde{X}$.
Notice that the lift of $\tau$ depends not only on $\tau$, but also
on the choice of $F$. When $\tau$ is a geodesic, we shall always
choose $F$ such that $\tilde{\tau}$ is also a geodesic.

There is a natural homeomorphism from $S^{1}$ to the boundary of
$\tilde{X}$. Set $A=\{(x,x):x\in S^{1}\}$. Define an equivalence
relation $\sim$ on $S^{1}\times S^{1}$ by identifying each $(x,y)$
with $(y,x)$. Then the space $G(\tilde{X})$ of geodesics of $\tilde{X}$
is homeomorphic to $(S^{1}\times S^{1}-A)/\sim$.

We generalize the notion of ``spiral'' as follows. Let $\tau,\rho$
be two distinct geodesics in $X$. We say that $\tau$ \emph{spirals
towards} $\rho$ if there exist a lift $\tilde{\rho}$ of $\rho$
to $\tilde{X}$ and a sequence of lifts $\{\tilde{\tau}_{k}\}_{k=1}^{\infty}$
of $\tau$ to $\tilde{X}$ such that $\tilde{\tau}_{k}$ tends to
$\tilde{\rho}$ as $k\rightarrow\infty$. For example, if $\rho$
is a closed geodesic and there exist a lift $\tilde{\rho}$ of $\rho$
and a lift $\tilde{\tau}$ of $\tau$ such that $\tilde{\rho},\tilde{\tau}$
have a common ideal endpoint, then $\tau$ spirals towards $\rho$.

Notice that if $\tau$ is a simple geodesic, $\rho$ is a simple closed
geodesic, and $\tau$ spirals towards $\rho$, then there exist a
lift $\tilde{\tau}$ of $\tau$ and a lift $\tilde{\rho}$ of $\rho$
such that $\tilde{\tau},\tilde{\rho}$ have a common ideal endpoint.
Therefore no confusion arises in the following lemma.
\begin{lem}[{\cite[Corollary 3.8.8]{hubbard2006teichmuller}}]
\label{lem:collar}Let $X$ be a closed hyperbolic surface and $\gamma$
a simple closed geodesic on $X$. Then any simple geodesic on $X$
that enters the collar $C_{\gamma}$ either intersects $\gamma$ or
spirals towards $\gamma$.
\end{lem}

A compact hyperbolic surface with totally geodesic boundary is called
\emph{a pair of pants} if it is homeomorphic to a disk with two holes.
\begin{thm}[{\cite[Theorem 3]{liu2019geodesic}}]
\label{thm:LY}A bijection between the interiors of two pairs of
pants is an isometry if it and its inverse map each geodesic onto
some geodesic.
\end{thm}

Let $R$ be a compact hyperbolic surface with boundary and $S$ the
interior of $R$. A set $\mathcal{L}$ of mutually disjoint simple
closed geodesics in $S$ is called \emph{maximal} if it intersects
every simple closed geodesic in $S$. We call a maximal set $\mathcal{L}=\{L_{j}\}_{j=1}^{N}$
of mutually disjoint simple closed geodesics in $S$ a \emph{system
of decomposing curves}, and the family $\mathcal{P}=\{P_{k}\}_{k=1}^{M}$
consisting of all connected components of $S-\bigcup_{j=1}^{N}L_{j}$
the \emph{pants decomposition} of $R$ (or $S$) corresponding to
$\mathcal{L}$.
\begin{prop}
\label{prop:pants_decomposition}Let $\mathcal{L}=\{L_{j}\}_{j=1}^{N}$
be an arbitrary system of decomposing curves on a compact hyperbolic
surface $R$ of genus $g$ with $n$ geodesic boundary components,
and let $\mathcal{P}=\{P_{k}\}_{k=1}^{M}$ be the pants decomposition
of $R$ corresponding to $\mathcal{L}$. Then $M$ and $N$ satisfy
\[
N=3g+n-3\quad\text{and}\quad M=2g+n-2.
\]
Therefore the area of $R$ is $(4g+2n-4)\pi$.
\end{prop}

In the literature, the notion of ``fill'' is usually only considered
for geodesics. Here we generalize it to general subsets. We say a
subset of $S$ \emph{fills} $S$ if it intersects every geodesic in
$S$.

We use the Klein disk model $\Delta$ for hyperbolic plane. In this
model, geodesics are straight Euclidean segments.

The following theorem is essential to understand the behavior of geodesics
in hyperbolic surfaces, though we will not mention it in the proofs
since it is used too frequently.
\begin{thm}[Gauss-Bonnet]
In the hyperbolic plane, an $n$-gon with angles $\alpha_{1},\cdots,\alpha_{n}$
has area $(n-2)\pi-(\alpha_{1}+\cdots+\alpha_{n})$.
\end{thm}

\subsection{Arc and curve complex}

Let $\overline{X}$ be a closed surface of genus $g$ and $B=\{p_{1},\cdots,p_{n}\}\subset\overline{X}$.
Then $X=\overline{X}-B$ is a surface of genus $g$ with $n$ punctures.

Let $\alpha:[0,1]\rightarrow\overline{X}$ be a path such that $\alpha(0),\alpha(1)\in B$.
We call $\tau=\alpha((0,1))$ an \emph{arc} in $X$ if $\tau\subset X$.
We say $\tau$ is \emph{essential} if $\alpha$ is not nullhomotopic
relative to $B$.

We say that two arcs $\tau_{0},\tau_{1}$ in $X$ are \emph{homotopic}
if there exists a continuous map $H:[0,1]\times[0,1]\rightarrow\overline{X}$
such that
\begin{itemize}
\item $H((0,1)\times\{0\})=\tau_{0}$ and $H((0,1)\times\{1\})=\tau_{1}$
\item $H(0,t),H(1,t)\in B$ for all $t\in[0,1]$
\item $H((0,1)\times[0,1])\subset X$.
\end{itemize}
The \emph{arc and curve complex} of $X$, denoted by $AC(X)$, is
the abstract simplicial complex whose $k$-simplices, for each $k\geqslant0$,
are collections of distinct homotopy classes of $k+1$ disjoint and
mutually non-homotopic one-dimensional submanifolds which can be either
essential simple closed curves or essential simple arcs in $X$. We
will denote by $|AC(X)|$ the set of vertices of $AC(X)$.
\begin{thm}[see \cite{korkmaz2010arc}]
\label{thm:KP}Every simplicial automorphism of $AC(X)$ is induced
by the action of a self-homeomorphism of $X$.
\end{thm}

See \cite{korkmaz2010arc} for more details on arc and curve complex.

\subsection{Elementary move}

In \cite{mosher1988tiling}, Mosher introduced the notion of elementary
moves on ideal triangulations. Let $X$ be a punctured surface of
finite type. Let $\delta$ be an \emph{ideal triangulation} of $X$,
i.e., a cell division of $X$ such that each 1-cell of $\delta$ is
a simple arc and each 2-cell of $\delta$ is a triangle. If $h$ is
an edge of $\delta$ that bounds two distinct triangles of $\delta$,
then these two triangles are joined along $h$ to form a quadrilateral;
removing $h$ and inserting the opposite diagonal of this quadrilateral,
we obtain a new ideal triangulation $\delta'$ which is said to be
obtained from $\delta$ by an \emph{elementary move}. See \cite{mosher1988tiling}
for more precise definitions.

Furthermore, If $X$ is equipped with a hyperbolic metric, and the
edges of $\delta$ and $\delta'$ are geodesics, then we say the elementary
move is \emph{geodesic}.
\begin{thm}[Connectivity Theorem for Elementary Moves, in \cite{mosher1988tiling}]
\label{thm:Mosher}Any two ideal triangulations of $X$ are related
by a finite sequence of elementary moves.
\end{thm}

\section{Outline of proof}

Throughout the rest of this article we fix a compact hyperbolic surface
$R$ of genus $g$ with $n$ geodesic boundary components, where $(g,n)\neq(0,3)$.
We will denote by $S$ the interior of $R$. We also fix a bijection
$f:S\rightarrow S$ such that $f$ and $f^{-1}$ map each geodesic
onto some geodesic.

Let $\tau$ be a geodesic in $S$. For convenience we introduce the
following concepts:

(1) We say $\tau$ \emph{cofills}\footnote{In this article, we only need the notion of ``strongly cofill'' but
not ``cofill''. Nevertheless, we introduce the notion of ``cofill''
because it is important, since a geodesic in a hyperbolic surface
of finite type is simple if and only if it cofills the surface.} $S$ if the union of all geodesics disjoint from $\tau$ intersects
every geodesic in $S$ except $\tau$.

(2) We say $\tau$ \emph{strongly cofills} $S$ if for any geodesic
$\rho\neq\tau$ in $S$ there exists a geodesic other than $\rho$
that is disjoint from $\tau$ but intersects $\rho$.

(3) We say $\tau$ is \emph{strongly pre-spiraled} if there exists
a geodesic $\rho$ disjoint from $\tau$ such that every geodesic
other than $\tau$ that intersects $\tau$ will intersect $\rho$
infinitely many times.

(4) We say $\tau$ is \emph{recurrent} if every geodesic that intersects
$\tau$ will intersect $\tau$ infinitely many times.

(5) We say $\tau$ is \emph{self-isolated} if for any $x\in\tau$
there exists a neighborhood $U$ of $x$ such that $(U,U\cap\tau)$
is homeomorphic to (disk, diameter).

(6) We say $\tau$ is an \emph{association} between the boundaries
if $\tau$ is non-closed and hits or spirals towards some boundary
component of $R$ in each direction.

It is easy to check that the first four properties are preserved by
$f$.

Section \ref{sec:A-preliminary-classification} is devoted to establishing
Table \ref{tab:classification_geodesics}. Notice that in the first
column of Table \ref{tab:classification_geodesics} we have listed
all the geodesics in $S$.

\begin{table}[h]
\noindent \begin{centering}
\begin{tabular}{|>{\centering}p{0.38\textwidth}|c|c|c|}
\hline 
a geodesic $\tau$ in $S$ & strongly cofills $S$ & recurrent & strongly pre-spiraled\tabularnewline
\hline 
\hline 
self-intersects & No &  & \tabularnewline
\hline 
simple and spirals towards another simple geodesic in $S$ & No &  & \tabularnewline
\hline 
simple and is an association between the boundaries & Yes & No & No\tabularnewline
\hline 
simple closed & Yes & No & Yes\tabularnewline
\hline 
simple but not self-isolated &  & Yes & \tabularnewline
\hline 
\end{tabular}
\par\end{centering}
\caption{\label{tab:classification_geodesics}A preliminary classification
of geodesics in $S$}
\end{table}

From now on we denote the set of simple closed geodesics in $S$ by
$G_{c}$ and the set of simple geodesics in $S$ that are associations
between the boundaries by $G$. It follows from Table \ref{tab:classification_geodesics}
that $G_{c}$ and $G$ are invariant under $f$.

The set $G$ is the disjoint union of three subsets $G(p),G(\infty),G(p,\infty)$,
where:
\begin{itemize}
\item $\tau\in G(p)\Leftrightarrow\tau$ hits some boundary component of
$R$ in each direction;
\item $\tau\in G(\infty)\Leftrightarrow\tau$ spirals towards some boundary
component of $R$ in each direction;
\item $\tau\in G(p,\infty)\Leftrightarrow\tau$ hits some boundary component
of $R$ in one direction and spirals towards some boundary component
of $R$ in another direction.
\end{itemize}
In Section \ref{sec:Topological-type}, we characterize the elements
of $G(\infty)$ and $G(p,\infty)$ in terms of their intersections
with the elements of $G$. Since $G$ is invariant under $f$, we
conclude that $G(\infty)$ and $G(p,\infty)$ are invariant under
$f$, and hence that $G(p)$ is invariant under $f$.

Furthermore, we will show that $f$ ``induces'' a simplicial automorphism
of $AC(S)$ in some sense. Based on this observation, we find a homeomorphism
$h:S\rightarrow S$ such that $f(\tau)$ is homotopic to $h(\tau)$
for any $\tau\in G(p)$. We regard $h$ as the ``topological type''
of $f$. In Section \ref{sec:Isometry}, we use $h$ to show that
the preimage of any neighborhood of $f(\tau)$, under $f$, is a neighborhood
of $\tau$ for any $\tau\in G(p)$. As any $x\in S$ is the intersection
of some geodesics $\tau_{1},\tau_{2}\in G(p)$, we see that $f$ is
continuous. By applying the same argument to $f^{-1}$, we see that
$f^{-1}$ is also continuous, and hence that $f$ is a homeomorphism.
Therefore $f$ is an isometry by Proposition \ref{prop:isometry}.
Finally, we generalize this to prove Theorem \ref{Theorem:main} in
Section \ref{sec:Proof-of-main}.

\section{\label{sec:A-preliminary-classification}A preliminary classification
of geodesics}

We first check that in the first column of Table \ref{tab:classification_geodesics}
we have listed all the geodesics in $S$:
\begin{prop}
Let $\tau$ be a simple self-isolated geodesic in $S$. Then exactly
one of the following cases occurs:

(1) $\tau$ is a simple closed geodesic;

(2) $\tau$ is an association between boundaries;

(3) $\tau$ spirals towards another simple geodesic in $S$.
\end{prop}

\begin{proof}
Suppose that $\tau$ is neither a simple closed geodesic nor an association
between boundaries. Then $\tau$ neither hits nor spirals towards
any boundary component of $R$ in one direction.

If $\tau$ does not hit $\partial R$ in the other direction, then
$\tau$ is a simple geodesic in the double $DR$ of $R$, where $DR$
is a closed surface obtained by taking two copies of $R$ and identifying
corresponding boundary components.

\begin{figure}[h]
\noindent \begin{centering}
\includegraphics{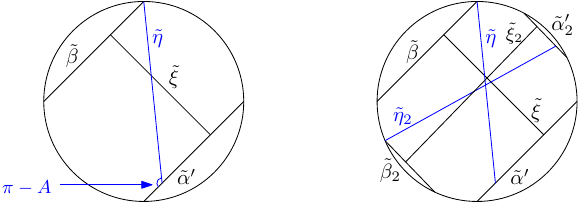}
\par\end{centering}
\caption{\label{fig:double_geodesic}If $\eta$ is a geodesic with self-intersection,
then there exists another lift $\tilde{\alpha}_{2}'\cup\tilde{\beta}_{2}\cup\tilde{\xi}_{2}\cup\tilde{\eta}_{2}$
of $\alpha'\cup\beta\cup\xi\cup\eta$ such that $\tilde{\eta}$ intersects
$\tilde{\eta}_{2}$. Now $\tilde{\xi}$ must intersect $\tilde{\xi}_{2}$,
which contradicts the fact that $\xi$ is simple.}
\end{figure}

If $\tau$ hits $\partial R$ in the other direction, we can adjust
$DR$ by a Fenchel-Nielsen deformation such that there exists a simple
geodesic $\overline{\tau}$ in $DR$ satisfying $\tau=\overline{\tau}\cap S$.
Let $DR=R\cup R'$, where $R'$ is isometric to $R$. Assume that
$\tau$ hits the boundary component $\alpha\subset\partial R$ at
an angle $A$. Let $\alpha'\subset\partial R'$ be the boundary component
glued to $\alpha$. We choose $\xi$ a simple geodesic segment in
$R'$ that connects $\alpha'$ to a simple closed geodesic $\beta$
in the interior of $R'$. Let $\tilde{\alpha}'\cup\tilde{\beta}\cup\tilde{\xi}$
be a lift of $\alpha'\cup\beta\cup\xi$ to the universal covering
space $\tilde{R}'\subset\Delta$. As illustrated in the left of Figure
\ref{fig:double_geodesic}, we choose a geodesic $\tilde{\eta}$ sharing
a common ideal endpoint with $\tilde{\beta}$ such that $\tilde{\eta}$
and $\tilde{\alpha}'$ meet at an angle $\pi-A$. Let $\eta$ denote
the projection of $\tilde{\eta}$ to $R'$. Since $\xi$ is simple,
$\eta$ is also simple. We cut $DR$ along $\alpha=\alpha'$ and glue
it back such that the endpoint of $\tau$ is identified with the endpoint
of $\eta$. Now $\overline{\tau}=\tau\cup\eta$ is a simple geodesic
in (new) $DR$.

In all cases above Lemma \ref{lem:collar} tells us that, for any
simple closed geodesic $\gamma$ in $S$, if $\tau$ enters the collar
$C_{\gamma}$, then $\tau$ either intersects $\gamma$ or spirals
towards $\gamma$.

\begin{figure}[h]
\noindent \begin{centering}
\includegraphics{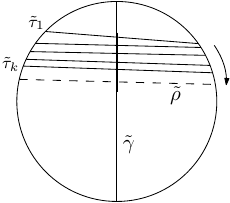}
\par\end{centering}
\caption{\label{fig:spiral}$\tau_{k}\rightarrow\rho$.}
\end{figure}

Suppose that $\tau$ does not spiral towards any simple \emph{closed}
geodesic in $S$. Using an arbitrary pants decomposition of $R$,
we see that there exists a simple closed geodesic $\gamma$ in $S$
that intersects $\tau$ infinitely many times. Let $\tilde{\gamma}$
be a lift of $\gamma$ to the universal covering space, then there
exists a compact subset of $\tilde{\gamma}$ that intersects infinitely
many lifts $\{\tilde{\tau}_{k}\}_{k=1}^{\infty}$ of $\tau$. Since
$\tau$ is simple, all the lifts of $\tau$ are mutually disjoint.
Thus there exists a subsequence of $\{\tilde{\tau}_{k}\}_{k=1}^{\infty}$
that tends to some geodesic $\tilde{\rho}$.

We claim that the projection $\rho$ of $\tilde{\rho}$ to $S$ is
a simple geodesic. Otherwise, there exists a lift $\tilde{\rho}'$
of $\rho$ that intersects $\tilde{\rho}$ in one point. There also
exists a sequence of lifts $\{\tilde{\tau}_{k}'\}_{k=1}^{\infty}$
of $\tau$ that tends to $\tilde{\rho}'$. Since $\tilde{\rho}$ intersects
$\tilde{\rho}'$, we see that $\tilde{\tau}_{k}$ intersects $\tilde{\tau}_{k}'$
for some $k$, a contradiction. Hence $\rho$ is simple.

Since $\tau$ is self-isolated, $\rho\neq\tau$. Therefore $\tau$
spirals towards the simple geodesic $\rho$ in $S$.
\end{proof}

Next, we prove five lemmas to establish Table \ref{tab:classification_geodesics}
line by line.
\begin{lem}
Let $\tau$ in $S$ be a geodesic with self-intersection. Then $\tau$
does not strongly cofill $S$, i.e., there exists a geodesic $\rho\neq\tau$
such that every geodesic $\alpha\neq\rho$ disjoint from $\tau$ is
disjoint from $\rho$.
\end{lem}

\begin{proof}
\begin{figure}[h]
\noindent \begin{centering}
\includegraphics{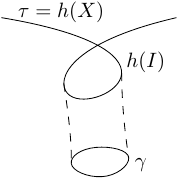}
\par\end{centering}
\caption{\label{fig:self_intersect_circle}A simple self-intersection.}
\end{figure}
As illustrated in Figure \ref{fig:self_intersect_circle}, we assume
that a self-intersection of $\tau$ arises when $\tau$ wraps around
a simple closed geodesic $\gamma$ in $R$. More precisely, Let $h:X\rightarrow S$
be a locally shortest curve such that $h(X)=\tau$, where $X=S^{1}$
if $\tau$ is closed, and otherwise $X=(0,1)$. Then there exists
an interval $I\subset X$ such that $h:I\rightarrow S$ is injective
and $h(I)$ is a simple closed curve freely homotopic to $\gamma$.

Let $A_{\gamma}$ denote the annulus bounded by $h(I)$ and $\gamma$.
Let $\tilde{A}_{\gamma}$ be a connected component of the preimage
of $A_{\gamma}$ under the universal covering map. Then $\tilde{A}_{\gamma}$
is bounded by a lift $\tilde{\gamma}$ of $\gamma$ and a lift $\xi$
of $h(I)$. Let $\tilde{\tau}$ be a lift of $\tau$ that intersects
$\xi$. Notice that $\xi$ and $\tilde{\tau}$ are contained in the
preimage of $\tau$ under the universal covering map.

\begin{figure}[h]
\noindent \begin{centering}
\includegraphics{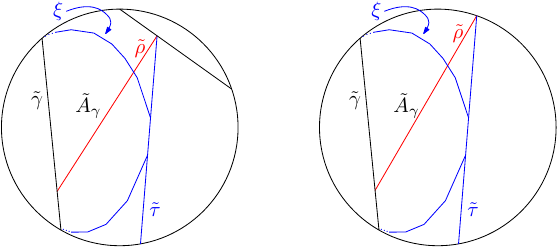}
\par\end{centering}
\caption{\label{fig:self_intersect_not_cofill}Construction of $\tilde{\rho}$.}
\end{figure}

If $\gamma\subset S$, the geodesic $\rho=\gamma$ has the desired
property.

If $\gamma\subset\partial R$, we choose $\tilde{\rho}$ passing through
$\tilde{A}_{\gamma}$ and joining an endpoint of $\tilde{\tau}$ to
$\gamma$; see Figure \ref{fig:self_intersect_not_cofill}. Let $\rho$
be the projection of $\tilde{\rho}$ to $S$. Suppose $\alpha\neq\rho$
is a geodesic in $S$ that intersects $\rho$. Let $\tilde{\alpha}$
be a lift of $\alpha$ that intersects $\tilde{\rho}$, then $\tilde{\alpha}$
must intersect $\xi\cup\tilde{\tau}$. It follows that $\rho$ would
be the desired geodesic if $\rho\neq\tau$. By adjusting the point
at which $\rho$ hits $\gamma$, we can now guarantee that $\rho\neq\tau$,
and the proof is complete.
\end{proof}

\begin{lem}
Let $\tau,\rho$ be two distinct simple geodesics in $S$ such that
$\tau$ spirals towards $\rho$. Then $\tau$ does not strongly cofill
$S$.
\end{lem}

\begin{proof}
Suppose that $\alpha\neq\rho$ is a geodesic that intersects $\rho$.
Let $\tilde{\alpha},\tilde{\rho}$ be lifts of $\alpha,\rho$ to the
universal covering space, respectively, such that $\tilde{\alpha}$
intersects $\tilde{\rho}$. Since $\tau$ spirals towards $\rho$,
there exists a sequence of lifts $\{\tilde{\tau}_{k}\}_{k=1}^{\infty}$
of $\tau$ that tends to $\tilde{\rho}$. It follows that $\tilde{\alpha}$
must intersect some $\tilde{\tau}_{k}$, hence $\alpha$ intersects
$\tau$. Therefore $\tau$ does not strongly cofill $S$.
\end{proof}

\begin{lem}
\label{lem:G}Let $\tau$ be a simple geodesic in $S$. Suppose that
$\tau$ is an association between boundaries. Then:

(i) $\tau$ strongly cofills $S$;

(ii) $\tau$ is not recurrent;

(iii) $\tau$ is not strongly pre-spiraled.
\end{lem}

\begin{proof}
\begin{figure}[h]
\noindent \begin{centering}
\includegraphics{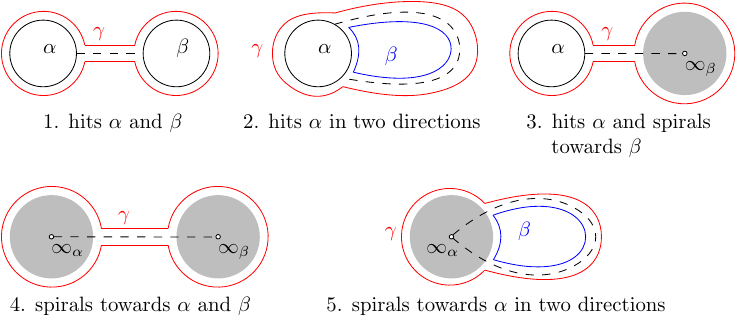}
\par\end{centering}
\caption{\label{fig:geodesic_pants}Every dashed line is $\tau$. The symbol
``$\infty$'' stands for ``spirals towards''.}
\end{figure}

We first show that $\tau$ is contained in a pair of pants $P\subset R$.
All the possibilities of $\tau$ have been listed in Figure \ref{fig:geodesic_pants}.
Here we deal only with the third case, since the other cases are similar.
Suppose that $\tau$ hits $\alpha$ at $p$ and spirals towards $\beta$,
where $\alpha,\beta$ are boundary components of $R$. Let $C_{\beta}$
be the collar around $\beta$ and $C_{0}$ the boundary component
of $C_{\beta}$ that intersects $\tau$. Write $\tau_{0}=\tau-C_{\beta}$
and $q=\tau\cap C_{0}$. Starting at $p$, going around $\alpha$
once in the clockwise direction, then walking along $\tau_{0}$ up
to $q$, next going around $C_{0}$ once in the clockwise direction,
and finally walking along $\tau_{0}$ back to $p$, we obtain a closed
curve freely homotopic to a simple closed geodesic $\gamma$. Then
$\tau$ is contained in the pair of pants $P$ whose boundary components
are $\alpha$, $\beta$ and $\gamma$.

\textbf{Proof of (i):} Without loss of generality, we may assume $\gamma\subset S$.
Extend $P$ to a pants decomposition $\mathcal{P}=\{P_{1}=P,P_{2},\cdots,P_{M}\}$
of $S$ and denote the corresponding system of decomposing curves
by $\{\xi_{1}=\gamma,\xi_{2},\cdots,\xi_{N}\}$. Let the geodesics
$\eta_{i1},\eta_{i2},\eta_{i3}$ be the edges of an ideal triangulation
of $P_{i}$, $i=2,\cdots,M$. We choose two simple geodesics $\rho_{1}$
and $\rho_{2}$ in $P$ as illustrated in Figure \ref{fig:association_strongly_cofill}.
Then 
\[
\sigma=\rho_{1}\cup\rho_{2}\cup\left(\bigcup_{i=1}^{N}\xi_{i}\right)\cup\left(\bigcup_{i=2}^{M}(\eta_{i1}\cup\eta_{i2}\cup\eta_{i3})\right)
\]
intersects every geodesic in $S$ except $\tau$.

\begin{figure}[h]
\noindent \begin{centering}
\includegraphics{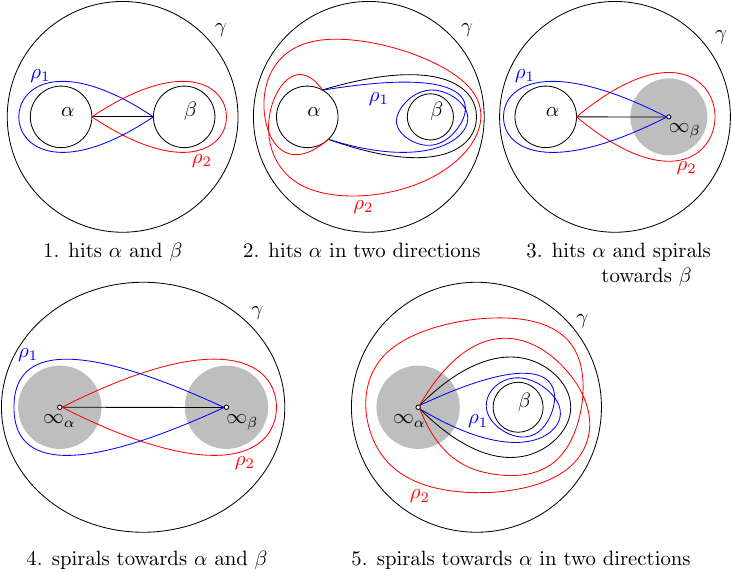}
\par\end{centering}
\caption{\label{fig:association_strongly_cofill}The symbol ``$\infty$'' stands
for ``spirals towards''.}
\end{figure}

Suppose, contrary to the assertion, that there exists a geodesic $\rho\neq\tau$
in $S$ such that every geodesic $\mu\neq\rho$ disjoint from $\tau$
are disjoint from $\rho$. Then $\rho$ intersects $\sigma$. Since
$\sigma$ is disjoint from $\tau$, we see that $\rho$ is contained
in $\sigma$. It is clear that $\rho$ can only be $\rho_{1}$ or
$\rho_{2}$. In cases 1, 3, 4 in Figure \ref{fig:association_strongly_cofill},
$\rho_{1}$ intersects $\rho_{2}$, so $\rho$ is not $\rho_{1}$
or $\rho_{2}$, a contradiction. In cases 2 and 5 in Figure \ref{fig:association_strongly_cofill},
$\rho_{1}$ and $\rho_{2}$ can be replaced by geodesics with more
self-intersections, so $\rho$ is not $\rho_{1}$ or $\rho_{2}$,
a contradiction. Hence $\tau$ must strongly cofill $S$.

\textbf{Proof of (ii):} Because $\tau$ is contained in the pair of
pants $P\subset R$ and $R$ is not a pair of pants, there exists
a simple closed geodesic $\rho$ in $S$ that intersects $\tau$ once
or twice. Thus $\tau$ is not recurrent.

\textbf{Proof of (iii):} In all cases, let $\lambda$ be the geodesic
representative of $[\alpha][\beta]^{-1}$; see Figure \ref{fig:association_not_prespiraled}.
Let $\rho$ be a geodesic in $S$ disjoint from $\tau$.

\begin{figure}[h]
\noindent \begin{centering}
\includegraphics{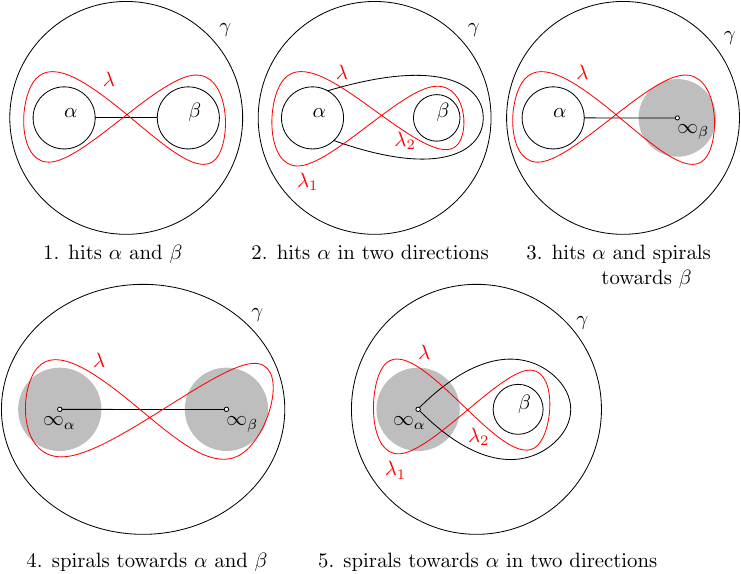}
\par\end{centering}
\caption{\label{fig:association_not_prespiraled}The symbol ``$\infty$'' stands
for ``spirals towards''.}
\end{figure}

In cases 1, 3, 4 in Figure \ref{fig:association_not_prespiraled},
since $\alpha$ and $\beta$ are boundary components of $R$, we see
that every time $\rho$ intersects $\lambda$, one of the following
situations must occur: $\rho$ hits $\alpha$, $\rho$ hits $\beta$,
$\rho$ spirals towards $\alpha$, or $\rho$ spirals towards $\beta$.
It follows that $\rho$ intersects $\lambda$ at most twice.

In cases 2 and 5 in Figure \ref{fig:association_not_prespiraled},
let $\lambda_{1},\lambda_{2}$ be simple closed curves freely homotopic
to $\alpha,\beta$, respectively, such that $\lambda=\lambda_{1}\cup\lambda_{2}$.
Every time $\rho$ intersects $\lambda_{1}$, it also hits $\alpha$
or spirals towards $\alpha$. Every time $\rho$ intersects $\lambda_{2}$,
it also intersects $\lambda_{1}$. It follows that $\rho$ intersects
$\lambda$ at most four times.

Therefore $\tau$ is not strongly pre-spiraled.
\end{proof}

\begin{lem}
Let $\tau$ be a simple closed geodesic in $S$. Then:

(i) $\tau$ strongly cofills $S$;

(ii) $\tau$ is not recurrent;

(iii) $\tau$ is strongly pre-spiraled.
\end{lem}

\begin{proof}
(i) Extend $\tau$ to a system of decomposing curves $\{\xi_{1}=\tau,\xi_{2},\cdots,\xi_{N}\}$
and denote the corresponding pants decomposition by $\{P_{1},\cdots,P_{M}\}$.
Let the geodesics $\eta_{i1},\eta_{i2},\eta_{i3}$ be the edges of
an ideal triangulation of $P_{i}$, $i=1,\cdots,M$. Then
\[
\sigma=\left(\bigcup_{i=2}^{N}\xi_{i}\right)\cup\left(\bigcup_{i=1}^{M}(\eta_{i1}\cup\eta_{i2}\cup\eta_{i3})\right)
\]
intersects every geodesic in $S$ except $\tau$. Extending $\tau$
to different systems of decomposing curves, we can see that each geodesic
contained in $\sigma$ intersects some other geodesic disjoint from
$\tau$. It follows that $\tau$ strongly cofills $S$.

(ii) Choose $\rho\neq\tau$ a closed geodesic intersecting $\tau$.
Then $\rho$ intersects $\tau$ at most finitely many times. Hence
$\tau$ is not recurrent.

(iii) There exists a geodesic $\rho$ that is disjoint from $\tau$
and spirals towards $\tau$, so $\tau$ is strongly pre-spiraled.
\end{proof}

\begin{lem}
Let $\tau$ be a simple geodesic in $S$ that is not self-isolated.
Then $\tau$ is recurrent.
\end{lem}

\begin{proof}
The preimage of $\tau$ under the universal covering map is the union
of a sequence of disjoint geodesics in the hyperbolic plane $\Delta$.
Since $\tau$ is not self-isolated, we see that there exists a sequence
of lifts $\{\tilde{\tau}_{k}\}_{k=0}^{\infty}$ of $\tau$ such that
$\tilde{\tau}_{k}$ tends to $\tilde{\tau}_{0}$ as $k\rightarrow\infty$.

Suppose that $\rho$ is a geodesic in $S$ that intersects $\tau$.
Let $\tilde{\rho}$ be a lift of $\rho$ that intersects $\tilde{\tau}_{0}$.
Then $\tilde{\rho}$ intersects $\tilde{\tau}_{k}$ for all $k$ sufficiently
large. Moreover, the intersection of $\tilde{\rho}$ with $\tilde{\tau}_{k}$
must tend to the intersection of $\tilde{\rho}$ with $\tilde{\tau}_{0}$
as $k\rightarrow\infty$. Since every covering map is a local homeomorphism,
it follows that $\rho$ intersects $\tau$ infinitely many times.
\end{proof}

\section{\label{sec:Topological-type}Topological type of bi-geodesic mapping}

\subsection{A classification of associations}

Recall that $G$ is the set of simple geodesics in $S$ that are associations
between boundaries, and $G=G(p)\cup G(\infty)\cup G(p,\infty)$.
\begin{lem}
\label{lem:G_infty}Suppose that $\tau\in G$. Then $\tau\in G(\infty)$
if and only if there exists a finite subset $\{\tau_{1},\cdots,\tau_{k}\}\subset G$
such that $\tau,\tau_{1},\cdots,\tau_{k}$ are disjoint geodesics
whose union intersects every geodesic in $S$.
\end{lem}

\begin{proof}
If $\tau\in G(\infty)$, by extending $\tau$ to the edges $\{\tau,\tau_{1},\cdots,\tau_{k}\}\subset G(\infty)$
of an ideal triangulation of $S$ we obtain the desired subset.

To prove the converse, we only need to show that there is no such
subset for any $\tau\in G(p)\cup G(p,\infty)$. Given $\tau\in G(p)\cup G(p,\infty)$,
let $\{\tau_{1},\cdots,\tau_{k}\}$ be an arbitrary finite subset
of $G$ such that $\tau,\tau_{1},\cdots,\tau_{k}$ are disjoint geodesics.
Assume that $\tau$ hits a boundary component $\alpha$. Then each
$\tau_{i}$ does not spiral towards $\alpha$. Thus $\alpha$ breaks
up into finitely many components when we cut $R$ along the geodesics
$\tau,\tau_{1},\cdots,\tau_{k}$. Hence there exists a geodesic $\tau_{k+1}\in G$
hitting $\alpha$ such that $\tau,\tau_{1},\cdots,\tau_{k+1}$ are
disjoint geodesics.
\end{proof}

In the case where $R$ has only one boundary component, we have $G=G(p)\cup G(\infty)$,
and the classification is done. For the other cases we have:
\begin{prop}
\label{prop:G_p_infty}Suppose that $R$ is a compact hyperbolic surface
of genus $g$ with $n\geqslant2$ geodesic boundary components. Let
$\tau\in G(p)\cup G(p,\infty)$ and $k=6g+3n-7$. Then $\tau\in G(p,\infty)$
if and only if there exists $k$ geodesics $\tau_{1},\cdots,\tau_{k}\in G(\infty)$
such that $\tau,\tau_{1},\cdots,\tau_{k}$ are disjoint geodesics.
\end{prop}

\begin{proof}
First assume $\tau\in G(p,\infty)$. Let $\alpha,\beta$ be two boundary
components of $R$ such that $\tau$ hits $\alpha$ and spirals towards
$\beta$. Choose $\tau_{1},\cdots,\tau_{j}\in G(\infty)$ such that
$\tau,\tau_{1},\cdots,\tau_{j}$ are disjoint geodesics whose union
intersects every element of $G(\infty)$; see Figure \ref{fig:join}.
By cutting $S$ along $\tau,\tau_{1},\cdots,\tau_{j}$ we obtain $l=(2j-1)/3$
ideal triangles and a quadrilateral $Q$ with two ideal vertices.
Since
\[
(4g+2n-4)\pi=\mathrm{Area}(S)=l\pi+\mathrm{Area}(Q)=\frac{2j-1}{3}\pi+[(4-2)\pi-\pi],
\]
we have $j=6g+3n-7=k$.

\begin{figure}[h]
\noindent \begin{centering}
\includegraphics{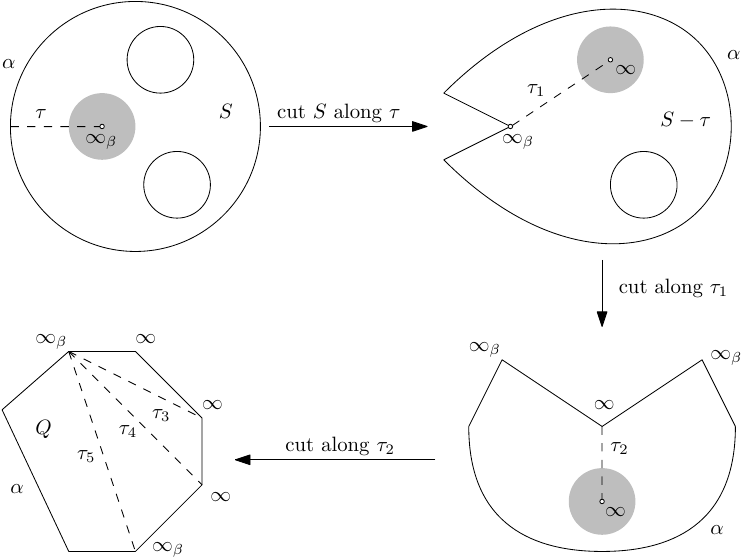}
\par\end{centering}
\caption{\label{fig:join}Here $(g,n)=(0,4)$ and $k=6\times0+3\times4-7=5$.}
\end{figure}

Conversely, assuming that $\tau\in G(p)$. Without loss of generality,
we may assume that $\tau$ hits two boundary components $\alpha$
and $\beta$. Let $\tau_{1},\cdots,\tau_{j}\in G(\infty)$ such that
$\tau,\tau_{1},\cdots,\tau_{j}$ are disjoint geodesics. We assume
that $j$ is maximal. Then $\tau\cup\tau_{1}\cup\cdots\cup\tau_{j}$
intersects every element of $G(\infty)$.

If $n=2$, then $g\geqslant1$, hence $j=0<6g-1=k$.

\begin{figure}[h]
\noindent \begin{centering}
\includegraphics{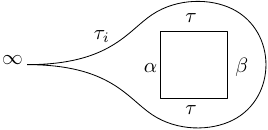}
\par\end{centering}
\caption{\label{fig:join2}The annulus.}
\end{figure}

Suppose that $n\geqslant3$. By cutting $S$ along $\tau,\tau_{1},\cdots,\tau_{j}$
we obtain $l=(2j-1)/3$ ideal triangles and an annulus; see Figure
\ref{fig:join2}. Let $P\subset R$ be the pair of pants that contains
$\tau$. Since none of $\tau_{1},\cdots,\tau_{j}$ enters $P$, we
have
\[
(4g+2n-4)\pi=\mathrm{Area}(R)\geqslant l\pi+\mathrm{Area}(P)=\frac{2j-1}{3}\pi+2\pi,
\]
hence $j\leqslant6g+3n-8.5<6g+3n-7=k$.

Therefore, if $\tau\notin G(p,\infty)$, we cannot find $k$ geodesics
$\tau_{1},\cdots,\tau_{k}\in G(\infty)$ such that $\tau,\tau_{1},\cdots,\tau_{k}$
are disjoint geodesics.
\end{proof}

\subsection{A realization of arc and curve complex}

Let $\alpha^{+}$ be a boundary component $\alpha$ equipped with
an orientation. Suppose that $\tau\in G(\infty)$ spirals towards
$\alpha$. We say $\tau$ \emph{spirals towards} $\alpha^{+}$ if
there exist a lift $\tilde{\alpha}^{+}$ of $\alpha^{+}$ and a lift
$\tilde{\tau}$ of $\tau$ such that the common ideal endpoint of
$\tilde{\alpha}^{+}$ and $\tilde{\tau}$ is the final point of $\tilde{\alpha}^{+}$;
see Figure \ref{fig:orientation}.

\begin{figure}[h]
\noindent \begin{centering}
\includegraphics{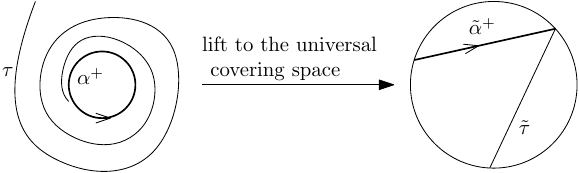}
\par\end{centering}
\caption{\label{fig:orientation}A geodesic spirals towards a simple closed
geodesic.}
\end{figure}

\begin{lem}
\label{lem:spiral_opposite}Suppose that $\tau_{1},\tau_{2}\in G(\infty)$
spiral towards the same boundary component $\alpha$ equipped with
opposite orientations. Then $\tau_{1}$ intersects $\tau_{2}$ infinitely
many times.
\end{lem}

\begin{proof}
Equip $\alpha$ with an appropriate orientation. There exist lifts
$\tilde{\alpha},\tilde{\tau}_{1},\tilde{\tau}_{2}$ of $\alpha,\tau_{1},\tau_{2}$,
respectively, such that the initial point of $\tilde{\alpha}$ is
an endpoint of $\tilde{\tau}_{1}$ and the final point of $\tilde{\alpha}$
is an endpoint of $\tilde{\tau}_{2}$; see Figure \ref{fig:spiral_boundary}.

\begin{figure}[h]
\noindent \begin{centering}
\includegraphics{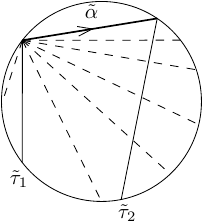}
\par\end{centering}
\caption{\label{fig:spiral_boundary}Iterations of $\tilde{\tau}_{1}$.}
\end{figure}
Consider a universal covering transformation $\gamma$ whose axis
is $\tilde{\alpha}$, then $\{\gamma^{k}(\tilde{\tau}_{1}):k\in\mathbb{Z}\}$
is a sequence of geodesics ending at the initial point of $\tilde{\alpha}$;
see the dashed lines in Figure \ref{fig:spiral_boundary}. It follows
that $\tau_{1}$ intersects $\tau_{2}$ infinitely many times.
\end{proof}

Throughout the rest of this section we denote the boundary components
of $R$ by $\alpha_{1},\cdots,\alpha_{n}$, and fix $\alpha_{i}^{+}$
the $\alpha_{i}$ equipped with an orientation, $i=1,\cdots,n$. Let
\[
G_{+\infty}\coloneqq\{\tau\in G(\infty):\text{if }\tau\text{ spirals towards }\alpha_{i}\text{, then }\tau\text{ spirals towards }\alpha_{i}^{+}\}.
\]

Notice that $S$ is topologically a surface of genus $g$ with $n$
punctures. Recall that $AC(S)$ is the arc and curve complex of $S$,
and $G_{c}$ is the set of simple closed geodesics in $S$. For any
$\beta\in|AC(S)|$, let $r(\beta)$ be its geodesic representative
in $G_{c}\cup G_{+\infty}$. We have:
\begin{lem}
\label{lem:AC_realization}The mapping $r:|AC(S)|\rightarrow G_{c}\cup G_{+\infty}$
is a well-defined bijection such that $\beta_{1},\beta_{2}\in|AC(S)|$
are joined by an edge of $AC(S)$ if and only if $r(\beta_{1})$ is
disjoint from $r(\beta_{2})$.
\end{lem}

\begin{proof}
\begin{figure}[h]
\noindent \begin{centering}
\includegraphics{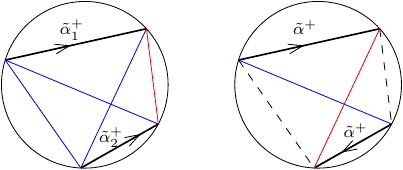}
\par\end{centering}
\caption{\label{fig:AC_arc}The left illustrates the case where $\beta$ spirals
towards two boundary components, while the right illustrates the case
where $\beta$ spirals towards one boundary component in two directions.
The red lines stand for $r(\beta)$.}
\end{figure}

Suppose that $\beta\in|AC(S)|$. If $\beta$ is the homotopy class
of a simple closed curve, then there exists a unique geodesic representative
$r(\beta)\in G_{c}$. If $\beta$ is the homotopy class of a simple
arc, then $\beta$ has two or four geodesic representatives in $G(\infty)$,
among which there exists a unique representative $r(\beta)\in G_{+\infty}$;
see Figure \ref{fig:AC_arc}. It follows that $r$ is well-defined.
Clearly $r$ is a bijection.

If $r(\beta_{1})$ is disjoint from $r(\beta_{2})$, from the definition
of $AC(S)$ we see that $\beta_{1}$ and $\beta_{2}$ are joined by
an edge of $AC(S)$.

Conversely, suppose that $r(\beta_{1})$ intersects $r(\beta_{2})$.
Here we deal only with the case where $r(\beta_{1}),r(\beta_{2})\in G_{+\infty}$,
since the other cases are similar.

\begin{figure}[h]
\noindent \begin{centering}
\includegraphics{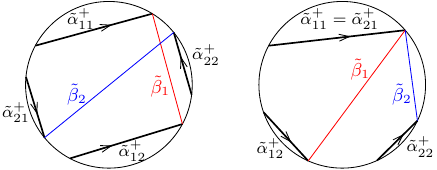}
\par\end{centering}
\caption{\label{fig:AC_arc_disjoint}If, for example, $\tilde{\alpha}_{11}=\tilde{\alpha}_{21}$,
then $\tilde{\beta}_{1}$ is disjoint from $\tilde{\beta}_{2}$, which
contradicts our hypothesis.}
\end{figure}

Let $\tilde{\beta}_{1}$ and $\tilde{\beta}_{2}$ be lifts of $r(\beta_{1})$
and $r(\beta_{2})$ to the universal covering space, respectively,
such that $\tilde{\beta}_{1}$ intersects $\tilde{\beta}_{2}$. For
$i=1,2$, we denote by $\tilde{\alpha}_{i1},\tilde{\alpha}_{i2}$
the geodesic boundaries sharing common ideal endpoints with $\tilde{\beta}_{i}$.
The geodesics $\tilde{\alpha}_{11},\tilde{\alpha}_{12},\tilde{\alpha}_{21},\tilde{\alpha}_{22}$
must be distinct geodesics; see Figure \ref{fig:AC_arc_disjoint}.
In addition, any two of $\tilde{\alpha}_{11},\tilde{\alpha}_{12},\tilde{\alpha}_{21},\tilde{\alpha}_{22}$
do not have a common endpoint, since their projection to $R$ are
geodesic boundaries.

Let $\gamma_{1}$ and $\gamma_{2}$ be arbitrary representatives of
$\beta_{1}$ and $\beta_{2}$, respectively. Then for $i=1,2$, there
exists a lift $\tilde{\gamma}_{i}$ of $\gamma_{i}$ that connects
$\tilde{\alpha}_{i1}$ to $\tilde{\alpha}_{i2}$. Here we allow each
endpoint of $\tilde{\gamma}_{i}$ to be an endpoint of $\tilde{\alpha}_{i1}$
or $\tilde{\alpha}_{i2}$. From the left of Figure \ref{fig:AC_arc_disjoint}
we see that $\tilde{\gamma}_{1}$ intersects $\tilde{\gamma}_{2}$,
hence $\gamma_{1}$ intersects $\gamma_{2}$. It follows that $\beta_{1}$
and $\beta_{2}$ are not adjacent in $AC(S)$.
\end{proof}

\begin{prop}
\label{prop:geodesic_elementary_move}(1) Let $\delta$ be an ideal
triangulation whose edges are elements of $G_{+\infty}$, and $\delta'$
an ideal triangulation obtained from $\delta$ by a geodesic elementary
move. Then the edges of $\delta'$ are elements of $G_{+\infty}$.

(2) Let $\delta,\delta'$ be two ideal triangulations whose edges
are elements of $G_{+\infty}$. Then $\delta$ and $\delta'$ are
related by a finite sequence of geodesic elementary moves.
\end{prop}

\begin{proof}
(1) Suppose that an edge $h$ of $\delta$ is replaced by a geodesic
$h'$ when we transform $\delta$ into $\delta'$ by the elementary
move. Suppose that $h'$ spirals towards a boundary component $\alpha$.
Then from the definition of elementary move we see that there exists
an edge $\tau\neq h$ of $\delta$ such that $\tau$ spirals towards
$\alpha$ and is disjoint from $h'$. Hence it follows from Lemma
\ref{lem:spiral_opposite} that $h'\in G_{+\infty}$.

(2) Theorem \ref{thm:Mosher} tells us that $\delta$ and $\delta'$
are related by a finite sequence of elementary moves. We can turn
these elementary moves into the desired geodesic ones by the mapping
$r$ in Lemma \ref{lem:AC_realization}.
\end{proof}

\subsection{\label{subsec:Homotopy-arc}Homotopy classes of simple arcs}

Recall that we fix a bijection $f:S\rightarrow S$ such that $f$
and $f^{-1}$ map each geodesic onto some geodesic. We conclude from
the discussion in Section \ref{sec:A-preliminary-classification}
that $G_{c}$ and $G$ are invariant under $f$, hence $G(\infty)$
is invariant under $f$ by Lemma \ref{lem:G_infty}, and therefore
$G(p,\infty)$ is invariant under $f$ by Proposition \ref{prop:G_p_infty}.
Since $G(p)=G-G(\infty)-G(p,\infty)$, we see that $G(p)$ is also
invariant under $f$.

Although $f$ is not defined on $\partial R$, we can say something
about the action of $f$ on $\partial R$ via $G(\infty)$:
\begin{lem}
There exists $\alpha_{i}^{+f}$ the boundary component $\alpha_{i}$
equipped with an orientation for $i=1,\cdots,n$, such that $f$ induces
a bijection
\[
f_{*}:G_{+\infty}\rightarrow G_{+\infty}^{f},\quad\tau\mapsto f(\tau),
\]
where $G_{+\infty}^{f}\coloneqq\{\tau\in G(\infty):$ if $\tau$ spirals
towards $\alpha_{i}$, then $\tau$ spirals towards $\alpha_{i}^{+f}\}$.
\end{lem}

\begin{proof}
Let $\tau_{1},\cdots,\tau_{m}\in G_{+\infty}$ be the edges of an
ideal triangulation $\delta$. Then $f(\tau_{1}),\cdots,f(\tau_{m})\in G(\infty)$
are the edges of the ideal triangulation $f(\delta)$. Hence for every
boundary component $\alpha$, there exists a geodesic $f(\tau_{i})$
that spirals towards $\alpha$. In addition, if $f(\tau_{i})$ and
$f(\tau_{j})$ spiral towards $\alpha_{k}$, then from Lemma \ref{lem:spiral_opposite}
we see that they spiral towards $\alpha_{k}$ equipped with the same
orientation, which we denote by $\alpha_{k}^{+f}$. In this way we
obtain $\alpha_{i}^{+f}$ for $i=1,\cdots,n$. Notice that $f(\tau_{1}),\cdots,f(\tau_{m})\in G_{+\infty}^{f}$.

Suppose that we transform $\delta$ into another ideal triangulation
$\delta'$ by a geodesic elementary move that replaces $\tau_{i}$
by a geodesic $\tau_{i}'$. Then $f(\tau_{i}')\cap f(\delta)=f(\tau_{i}')\cap f(\tau_{i})=f(\tau_{i}'\cap\tau_{i})$
is a point. Hence $f(\delta')$ is obtained from $f(\delta)$ by a
geodesic elementary move that replaces $f(\tau_{i})$ by $f(\tau_{i}')$.
Using a version of Proposition \ref{prop:geodesic_elementary_move}
for $G_{+\infty}^{f}$, we see that the edges of $f(\delta')$ are
elements of $G_{+\infty}^{f}$. In particular, $f(\tau_{i}')$ belongs
to $G_{+\infty}^{f}$.

Let $\tau$ be an arbitrary element of $G_{+\infty}$. We can extend
$\tau$ to an ideal triangulation $\delta_{\tau}$ whose edges are
elements of $G_{+\infty}$. Then Proposition \ref{prop:geodesic_elementary_move}
tells us that $\delta_{\tau}$ and $\delta$ are related by a finite
sequence of geodesic elementary moves. By the previous paragraph,
we have $f(\tau)\in G_{+\infty}^{f}$. It follows that $f_{*}:G_{+\infty}\rightarrow G_{+\infty}^{f}$
is a well-defined map.

Repeating the reasoning above, we see that
\[
(f^{-1})_{*}:G_{+\infty}^{f}\rightarrow G_{+\infty},\quad\tau\mapsto f^{-1}(\tau)
\]
is a well-defined map. It follows that $f_{*}:G_{+\infty}\rightarrow G_{+\infty}^{f}$
is a bijection.
\end{proof}

Extend $f_{*}$ to $G_{c}\cup G_{+\infty}$ by setting $f_{*}(\tau)=f(\tau)$
for $\tau\in G_{c}$. Then $f_{*}:G_{c}\cup G_{+\infty}\rightarrow G_{c}\cup G_{+\infty}^{f}$
is a bijection.

For any $\beta\in|AC(S)|$, let $r^{f}(\beta)$ be its geodesic representative
in $G_{c}\cup G_{+\infty}^{f}$. Similar to Lemma \ref{lem:AC_realization},
the mapping $r^{f}:|AC(S)|\rightarrow G_{c}\cup G_{+\infty}^{f}$
is a well-defined bijection such that $\beta_{1},\beta_{2}\in|AC(S)|$
are joined by an edge of $AC(S)$ if and only if $r^{f}(\beta_{1})$
is disjoint from $r^{f}(\beta_{2})$.

From the discussion above, we see that
\[
\varphi:|AC(S)|\rightarrow|AC(S)|,\quad\beta\mapsto(r^{f})^{-1}(f_{*}(r(\beta)))
\]
is a well-defined bijection such that $\beta_{1},\beta_{2}$ are adjacent
in $AC(S)$ if and only if $\varphi(\beta_{1}),\varphi(\beta_{2})$
are adjacent in $AC(S)$, and hence $\varphi$ can be naturally extended
to a simplicial automorphism of $AC(S)$. It follows from Theorem
\ref{thm:KP} that there exists a homeomorphism $h:S\rightarrow S$
such that $\varphi(\beta)=h(\beta)$ for any $\beta\in|AC(S)|$. We
regard $h$ as the ``topological type'' of $f$.

We have proved that, for every $\tau\in G_{c}\cup G_{+\infty}$, the
geodesic $f(\tau)$ belongs to $G_{c}\cup G_{+\infty}^{f}$ and is
homotopic to $h(\tau)$. We can naturally extend $h$ to a self-homeomorphism
of $R$. It follows immediately that $\tau\in G_{+\infty}$ spirals
towards $\alpha\subset\partial R$ if and only if $f(\tau)\in G_{+\infty}^{f}$
spirals towards $h(\alpha)\subset\partial R$.
\begin{lem}
\label{lem:homotopy}Suppose that $\tau_{1},\tau_{2}\in G(p)$. Then
$\tau_{1}$ and $\tau_{2}$ are homotopic if and only if $f(\tau_{1})$
and $f(\tau_{2})$ are homotopic.
\end{lem}

\begin{proof}
It is sufficient to show that if $\tau_{1}$ and $\tau_{2}$ are homotopic,
then $f(\tau_{1})$ and $f(\tau_{2})$ are homotopic.

As we see in the proof of Lemma \ref{lem:G}, $\tau_{1}$ is contained
in a pair of pants $P\subset R$. Since $\tau_{2}$ is homotopic to
$\tau_{1}$, it is also contained in $P$.

Extend $P$ to a pants decomposition of $S$ and denote the corresponding
system of decomposing curves by $\{\xi_{1},\cdots,\xi_{N}\}$. Because
$\tau_{1},\tau_{2}$ are disjoint from $\xi_{1}\cup\cdots\cup\xi_{N}$,
we see that $f(\tau_{1}),f(\tau_{2})$ are disjoint from $f(\xi_{1})\cup\cdots\cup f(\xi_{N})$.
Notice that $\{f(\xi_{1}),\cdots,f(\xi_{N})\}$ is also a system of
decomposing curves.

Assume that $\tau_{1}$ hits a boundary component $\alpha$. Choose
a geodesic $\tau\in G_{+\infty}$ that spirals towards $\alpha$ in
two directions. Then $\tau_{1}$ intersects $\tau$ infinitely many
times, hence $f(\tau_{1})$ intersects $f(\tau)$ infinitely many
times and hits the boundary component $h(\alpha)$. Since $\tau_{2}$
is homotopic to $\tau_{1}$, it also hits $\alpha$, and hence $f(\tau_{2})$
hits $h(\alpha)$ by the same reasoning as above. It follows that
$f(\tau_{1})$ hits the same boundary components as $f(\tau_{2})$.
Combining this with the discussion in the previous paragraph, we see
that $f(\tau_{1})$ and $f(\tau_{2})$ are contained in the same pair
of pants. Since the homotopy class of a simple arc in a pair of pants
is determined by the boundary components it hits, we conclude that
$f(\tau_{1})$ is homotopic to $f(\tau_{2})$.
\end{proof}

\section{\label{sec:Isometry}From bijection to isometry}
\begin{lem}
\label{lem:neighborhood}Suppose $\tau\in G(p)$ hits $\partial R$
at two points. Then for any $\varepsilon$-neighborhood $V(f(\tau))$
of $f(\tau)$ there is a neighborhood $U(\tau)$ of $\tau$ such that
$f(U(\tau))\subset V(f(\tau))$.
\end{lem}

\begin{proof}
Let $P\subset R$ be a pair of pants that contains $\tau$. Denote
the geodesic boundary components of $P$ by $\alpha,\beta,\gamma$.
Recall the homeomorphism $h:R\rightarrow R$ in Subsection \ref{subsec:Homotopy-arc}.
Denote the simple closed geodesics homotopic to $h(\alpha),h(\beta),h(\gamma)$
by $\alpha',\beta',\gamma'$, respectively. Let $P'$ be the pair
of pants whose boundary components are $\alpha',\beta',\gamma'$.
Then $f(\tau)$ is contained in $P'$. Without loss of generality,
we may assume that $\gamma$ is contained in $S$. The proof divides
into two cases.

\textbf{Case I}: Assume that $\tau$ connects $\alpha$ to $\beta$.

\begin{figure}[h]
\noindent \begin{centering}
\includegraphics{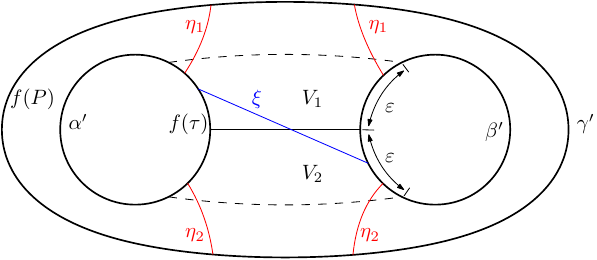}
\par\end{centering}
\caption{\label{fig:homeomorphism_2}The case that $\tau$ connects $\alpha$
to $\beta$.}
\end{figure}

We denote the connected components of $V(f(\tau))-f(\tau)$ by $V_{1}$
and $V_{2}$. Choose the geodesics $\xi,\eta_{1},\eta_{2}\in G(p)$
as illustrated in Figure \ref{fig:homeomorphism_2} such that:

(i) $\xi$ is contained in $V(f(\tau))$ and meets $f(\tau)$ at one
point,

(ii) $\xi$ is homotopic to $f(\tau)$,

(iii) $\eta_{1},\eta_{2}$ are geodesics that intersect $\gamma$
and connect $\alpha'$ to $\beta'$,

(iv) $\eta_{1}$ hits $\partial V_{1}\cap\partial R$, while $\eta_{2}$
hits $\partial V_{2}\cap\partial R$,

(v) $\xi\cup f(\tau)$ is disjoint from $\eta_{1}\cup\eta_{2}$.

Let $L$ be the set of geodesics $\rho\in G(p)$ satisfying:

(1) $\rho$ is homotopic to $f(\tau)$, and

(2) $\rho$ is disjoint from $\eta_{1}\cup\eta_{2}$.

\noindent Then $f(\tau)$ and $\xi$ belong to $L$, and every $\rho\in L$
is contained in $V(f(\tau))$.

Lemma \ref{lem:homotopy} tells us that $f^{-1}(L)$ is the set of
geodesics $\rho\in G(p)$ satisfying:

(1$'$) $\rho$ is homotopic to $\tau$, and

(2$'$) $\rho$ is disjoint from $f^{-1}(\eta_{1})\cup f^{-1}(\eta_{2})$.

\noindent Notice that $\tau$ and $f^{-1}(\xi)$ belong to $f^{-1}(L)$.
Since:

(i$'$) $f^{-1}(\xi)$ meets $\tau$ at one point,

(ii$'$) $f^{-1}(\xi)$ is homotopic to $\tau$,

(v$'$) $f^{-1}(\xi)\cup\tau$ is disjoint from $f^{-1}(\eta_{1})\cup f^{-1}(\eta_{2})$,

\noindent we see that $U(\tau)\coloneqq\{x\in S:x\in\rho$ for some
$\rho\in f^{-1}(L)\}$ is a neighborhood of $\tau$ such that $f(U(\tau))\subset V(f(\tau))$.

\textbf{Case II}: Assume without loss of generality that $\tau$ hits
$\alpha$ in two directions.

We first show that $f(\tau)$ hits $\partial R$ at two points. Suppose
to the contrary that $f(\tau)$ hits $\partial R$ at one point as
illustrated in Figure \ref{fig:distinct_endpoints}.

\begin{figure}[h]
\noindent \begin{centering}
\includegraphics{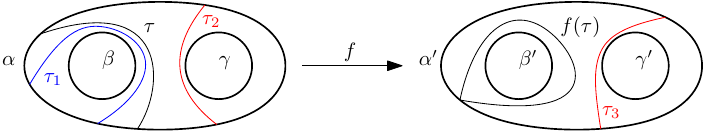}
\par\end{centering}
\caption{\label{fig:distinct_endpoints}What if $f(\tau)$ hits $\partial R$
only once?}
\end{figure}

Let $A=\{\rho\in G(p):\rho$ is homotopic to and disjoint from $\tau\}$.
We define a relation $\sim$ on $A$ as follows:

(a) If $\rho_{1}$ intersects $\rho_{2}$, then $\rho_{1}\sim\rho_{2}$;

(b) If $\rho_{1}\sim\rho_{2}$ and $\rho_{2}\sim\rho_{3}$, then $\rho_{1}\sim\rho_{3}$.

\noindent Clearly $\sim$ is an equivalence relation on $A$ and divides
$A$ into two equivalence classes.

The relation $\sim$ is preserved by $f$. More precisely, let $A'=\{\rho\in G(p):\rho$
is homotopic to and disjoint from $f(\tau)\}$ and define an equivalence
relation $\sim$ on $A'$ as follows:

(a$'$) If $\rho_{1}$ intersects $\rho_{2}$, then $\rho_{1}\sim\rho_{2}$;

(b$'$) If $\rho_{1}\sim\rho_{2}$ and $\rho_{2}\sim\rho_{3}$, then
$\rho_{1}\sim\rho_{3}$.

\noindent Then for any $\rho_{1},\rho_{2}\in A$, we have $\rho_{1}\sim\rho_{2}\Leftrightarrow f(\rho_{1})\sim f(\rho_{2})$.

However, $A'$ does not divide into two equivalence classes. In fact,
for any $\rho_{1},\rho_{2}\in A'$, we have $\rho_{1}\sim\rho_{2}$;
see the right of Figure \ref{fig:distinct_endpoints}. Thus we reach
a contradiction. Hence $f(\tau)$ must hit $\partial R$ at two points.

\begin{figure}[h]
\noindent \begin{centering}
\includegraphics{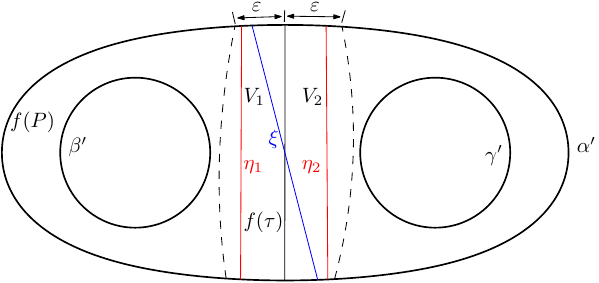}
\par\end{centering}
\caption{\label{fig:homeomorphism_1}The case that $\tau$ hits $\alpha$ twice.}
\end{figure}

As before, we denote the connected components of $V(f(\tau))-f(\tau)$
by $V_{1}$ and $V_{2}$. Choose the geodesics $\xi,\eta_{1},\eta_{2}\in G(p)$
such that:

(i) $\xi$ is contained in $V(f(\tau))$ and meets $f(\tau)$ at one
point,

(ii) $\xi,\eta_{1},\eta_{2}$ are homotopic to $f(\tau)$,

(iii) $\eta_{1}$ is contained in $V_{1}$, while $\eta_{2}$ is contained
in $V_{2}$,

(iv) $\xi\cup f(\tau)$ is disjoint from $\eta_{1}\cup\eta_{2}$.

Let $L$ be the set of geodesics $\rho\in G(p)$ satisfying:

(1) $\rho$ is homotopic to $f(\tau)$, and

(2) $\rho$ is disjoint from $\eta_{1}\cup\eta_{2}$ but intersects
$\xi\cup f(\tau)$.

\noindent Then $f(\tau)$ and $\xi$ belong to $L$, and every $\rho\in L$
is contained in $V(f(\tau))$.

Lemma \ref{lem:homotopy} tells us that $f^{-1}(L)$ is the set of
geodesics $\rho\in G(p)$ satisfying:

(1$'$) $\rho$ is homotopic to $\tau$, and

(2$'$) $\rho$ is disjoint from $f^{-1}(\eta_{1})\cup f^{-1}(\eta_{2})$
but intersects $f^{-1}(\xi)\cup\tau$.

\noindent Notice that $\tau$ and $f^{-1}(\xi)$ belong to $f^{-1}(L)$.
Since:

(i$'$) $f^{-1}(\xi)$ meets $\tau$ at one point,

(ii$'$) $f^{-1}(\xi),f^{-1}(\eta_{1}),f^{-1}(\eta_{2})$ are homotopic
to $\tau$,

(iv$'$) $f^{-1}(\xi)\cup\tau$ is disjoint from $f^{-1}(\eta_{1})\cup f^{-1}(\eta_{2})$,

\noindent we see that $U(\tau)\coloneqq\{x\in S:x\in\rho$ for some
$\rho\in f^{-1}(L)\}$ is a neighborhood of $\tau$ such that $f(U(\tau))\subset V(f(\tau))$.
This completes the proof.
\end{proof}

Let $D$ denote the union of those elements of $G(p)$ which hit $\partial R$
at two points. Then $f(D)=D$. It follows from Lemma \ref{lem:neighborhood}
that $f|_{D}$ is continuous. By applying the same argument to $f^{-1}$,
we see that $f^{-1}|_{D}$ is also continuous. We conclude that $f|_{D}$
is a self-homeomorphism of $D$.

Though it is not always the case that $D=S$, we observe that the
proof of Proposition \ref{prop:isometry} works for $f|_{D}:D\rightarrow D$.
Hence $f|_{D}$ is an isometry. Since every component of $S-\overline{D}$
is a domain, we can extend $f|_{D}$ to an isometry $h:S\rightarrow S$.

Suppose that $p\in S$. Then there is a family of geodesics $\{\tau_{i}\}_{i\in I}$
whose unique intersection is $p$. Since a geodesic in $S$ is uniquely
determined by its intersection with $D$, we see that $f(\tau_{i})=h(\tau_{i})$
for $i\in I$. Hence $f(p)=h(p)$. It follows that $f:S\rightarrow S$
is an isometry.

Since the proof works for bi-geodesic mapping $f:S_{1}\rightarrow S_{2}$,
we have:
\begin{thm}
\label{thm:same_topology_to_isometry}Let $R_{1},R_{2}$ be two compact
hyperbolic surfaces of genus $g$ with $n\geqslant1$ geodesic boundary
components. Let $S_{1}$ and $S_{2}$ be the interiors of $R_{1}$
and $R_{2}$, respectively. Suppose that $(g,n)\neq(0,3)$ and $f:S_{1}\rightarrow S_{2}$
is a bijection such that $f$ and $f^{-1}$ map each geodesic onto
some geodesic. Then $f$ is an isometry.
\end{thm}

\section{\label{sec:Proof-of-main}Proof of main theorem}
\begin{prop}
\label{prop:same_topology}Let $R_{i}$ be a compact hyperbolic surface
of genus $g_{i}$ with $n_{i}\geqslant1$ geodesic boundary components,
and $S_{i}$ the interior of $R_{i}$, $i=1,2$. Suppose that $f:S_{1}\rightarrow S_{2}$
is a bijection such that $f$ and $f^{-1}$ map each geodesic onto
some geodesic. Then $(g_{1},n_{1})=(g_{2},n_{2})$.
\end{prop}

\begin{proof}
Suppose that $(g_{1},n_{1})=(0,3)$, i.e., $R_{1}$ is a pair of pants.
Suppose to the contrary that $(g_{2},n_{2})\neq(0,3)$. Then there
exists a simple closed geodesic in $S_{2}$, i.e., there exists a
geodesic $\tau$ in $S_{2}$ that strongly cofills $S_{2}$ and is
strongly pre-spiraled but not recurrent. Hence $f^{-1}(\tau)$ is
a geodesic in $S_{1}$ that strongly cofills $S_{1}$ and is strongly
pre-spiraled but not recurrent, which is impossible. Thus $(g_{2},n_{2})$
must equal $(0,3)$.

Suppose that $(g_{1},n_{1})\neq(0,3)$. By Table \ref{tab:classification_geodesics},
we see that $\tau$ is a simple closed geodesic if and only if $f(\tau)$
is a simple closed geodesic. Since $f$ is a bijection, Proposition
\ref{prop:pants_decomposition} implies that
\begin{equation}
3g_{1}+n_{1}-3=3g_{2}+n_{2}-3.\label{eq:1}
\end{equation}

If $n_{1}=1$, then Proposition \ref{prop:G_p_infty} gives $n_{2}=1$,
and hence $(g_{1},n_{1})=(g_{2},n_{2})$ by \eqref{eq:1}.

If $n_{1}\geqslant2$, then Proposition \ref{prop:G_p_infty} gives
$n_{2}\geqslant2$ and
\[
6g_{1}+3n_{1}-7=6g_{2}+3n_{2}-7.
\]
Combining this with \eqref{eq:1} we also obtain $(g_{1},n_{1})=(g_{2},n_{2})$.
\end{proof}

Finally, Theorem \ref{Theorem:main} follows from Theorem \ref{thm:LY},
Theorem \ref{thm:same_topology_to_isometry} and Proposition \ref{prop:same_topology}.

\section{Concluding Remarks}

We now see that the bi-geodesic mappings between compact hyperbolic
surfaces with boundary are isometries. The method may be generalized
to apply to other manifolds \emph{with boundary}. In particular, a
paper on punctured hyperbolic surfaces with boundary will be published
elsewhere.

However, it seems that the method does not work for manifolds \emph{without
boundary}. It is still an open problem to determine whether the bi-geodesic
mappings between closed hyperbolic surfaces or punctured hyperbolic
surfaces are isometries.

Furthermore, in contrast with the Fundamental Theorem of Affine Geometry,
it is desirable to determine whether a bi-geodesic mapping is always
a \emph{geodesic mapping}, i.e., a bijection that maps each geodesic
onto some geodesic.

\subsection*{Acknowledgments.}

This work was partially supported by National Natural Science Foundation
of China (Grant No. 11771456) and Fundamental Research Funds for the
Central Universities (Grant No. 531118010617). The author is grateful
to Professor Lixin Liu for many useful suggestions and discussions.

\bibliographystyle{amsplain}
\bibliography{reference}

\end{document}